\documentclass[12pt]{article}
\usepackage{amssymb,latexsym}
\usepackage{tikz}
\usepackage{amsmath}
\usepackage{amsthm}
\usepackage{enumerate}
\newtheorem{theorem}{Theorem}[section]
\newtheorem{corollary}[theorem]{Corollary}

\newtheorem{example}[theorem]{Example}

\newtheorem{lemma}[theorem]{Lemma}
\newtheorem{proposition}[theorem]{Proposition}
\newtheorem{remark}[theorem]{Remark}

\usepackage{authblk}

\usepackage{tikz}

\begin{document}
\begin{sloppypar}

\title{An other approach of the diameter of $\Gamma(R)$ and $\Gamma(R[X])$}

\author{A. Cherrabi\thanks{Corresponding Author: azzouz.cherrabi@gmail.com}, H. Essannouni, E. Jabbouri,  A. Ouadfel
%\thanks{Corresponding Author: aliouadfel@gmail.com}
}
 \affil{Laboratory of Mathematics, Computing and Applications-Information Security (LabMia-SI)\\
 Faculty of Sciences, Mohammed-V University in Rabat.\\
 Rabat. Morocco.}
\date{}
\maketitle
\begin{abstract}Using the new extension of the zero-divisor graph $\widetilde{\Gamma}(R)$ introduced in
\cite{Groupe}, we give an  approach of the diameter of $\Gamma(R)$ and $\Gamma(R[X])$ other than given
in \cite{Lucas} thus we give a complete characterization for the possible diameters $1$, $2$ or $3$
of $\Gamma(R)$ and $\Gamma(R[x])$.
\end{abstract}

\section*{Introduction}
The idea of a zero-divisor graph was introduced by I. Beck in \cite{Beck} while he was mainly
interested in colorings. In beck's work, the graph  $\Gamma_0(R)$ associated with a
nontrivial commutative unitary ring $R$ is the undirected simple graph where the vertices are all
elements of $R$ and two vertices $x$ and $y$ are adjacent if and only if $xy=0$.\\
the study of the interaction between the properties of ring theory and the properties of graph
theory begun with the article of D.F. Anderson and P.S. Livingston  where  they modified
the graph considering the zero-divisor graph  $\Gamma(R)$ with vertices in
$Z(R)^{\star}=Z(R)\setminus \{0\}$, where $Z(R)$ is the set of  zero-divisors of $R$,
and for distinct $x,y \in Z(R)^{\star}$, the vertices $x$ and $y$ are adjacent if and only if $xy=0$ \cite{AnderLiv}. Also, D. F Anderson and A. Badawi introduced the total graph $T(\Gamma (R))$ of a commutative ring $R$ with all elements of $R$ as vertices and for distinct $x,y \in R$, the vertices $x$ and $y$ are adjacent if and only if $x+y \in Z(R)$ \cite{AnderBada}.

In \cite{Groupe}, we introduced a new graph, denoted  $\widetilde{\Gamma}(R)$,
as the undirected simple graph  whose vertices are the nonzero zero-divisors of $R$ and for distinct
$x,y\in Z(R)^{\star}$, $x$ and $y$ are adjacent if and only if $xy=0$ or $x+y\in Z(R)$. \\
Recall that a path $P$ in the graph $G=(V,E)$ is a finite sequence $(x_0,\dots,x_k)$ of  distinct
vertices such that for all $i=0,\dots,k-1$, $x_ix_{i+1}$ is an edge. In this case, we said
that $x_0$ and $x_k$ are linked by $P=x_0\-- x_k$ and the length of $P$ is $k$, i.e., the number of its edges.
$G$ is said to be connected if each pair of distincts vertices belongs to a path. Also, if $G$
has a path $x\-- y$, then the distance between $x$ and $y$, written $d_G(x,y)$ or simply $d(x,y)$ is
the least length of a $x\--y$ path. If $G$ has no such path, then $d(x,y)=\infty$. The diameter of $G$,
denoted $diam(G)$, is the greatest distance between any two vertices in $G$. A graph $G$ is
complete if each pair of distinct vertices forms an edge, i.e., if $diam(G)=1$. \\
$R$ is a nontrivial commutative unitary ring and general references for commutative
 ring theory are \cite{A} and \cite{Kap}.

In \cite{Lucas}, T. G. Lucas  has studied situations where $diam( \Gamma(R))$ and
$diam( \Gamma(R[x]))$ are $=1,2$ or $3$ and gave a complete characterization of these diameter
strictly in terms of properties of the ring $R$.

In this paper, we  give another approach of this problem using the properties of the new
graph $\widetilde{\Gamma}(R)$. In the first section, we  begin by showing the
link between the non-completeness of $\widetilde{\Gamma}(R)$ and the diameter of $diam( \Gamma(R))$ and
we give a complete characterization of the diameter of $diam( \Gamma(R))$ using  the nature of the ring $R$.
In the second section, we give some  examples illustrating cases where the diameter of
$\Gamma(R)$ is $1$, $2$ or $3$.
The third section is reserved for characterization of the $diam( \Gamma(R[X]))$ in terms
of the nature of the ring $R$.

We recall that  $|Z(R)^{\star}|=1$ if
and only if $R\simeq \mathbb{Z}_4$ or $R\simeq \mathbb{Z}[X]/(X^2)$ (cf. Example 2.1, \cite{AnderLiv}) so
we  assume, along this paper, that $R$ is such that $|Z(R)^{\star}|>1$.

\section{ diameter of $\Gamma (R)$}
This section is devoted to the study of diameter of $\Gamma(R)$. We begin by recalling the Lucas's result:
\begin{theorem}( cf. theorem 2.6, \cite{Lucas}) Let $R$ be a ring.
\begin{enumerate}[ (1)]
  \item  $diam(\Gamma(R))=0$ if and only if R is (nonreduced and) isomorphic to either $\mathbb{Z}_4$ or
$\mathbb{Z}_2[y]/(y^2)$.
  \item $diam(\Gamma(R))=1$ if and only if $xy=0$ for each distinct pair of zero divisors and $R$ has
at least two nonzero zero divisors.
  \item $diam(\Gamma(R))=2$  if and only if either (i) $R$ is reduced with exactly two minimal primes
and at least three nonzero zero divisors, or (ii) $Z(R)$ is an ideal whose square is not
$(0)$ and each pair of distinct zero divisors has a nonzero annihilator.
  \item $diam(\Gamma(R))=3$ if and only if there are zero divisors $a\neq b$ such that $(0:(a,b))=(0)$
and either (i) $R$ is a reduced ring with more than two minimal primes, or (ii) $R$ is
nonreduced.
\end{enumerate}
\end{theorem}

\begin{remark}As stated above, we assume that $R$ is such that $diam(\Gamma(R))\neq 0$, i.e.,
$R\not\simeq \mathbb{Z}_4$ and $R\not\simeq \mathbb{Z}_2[X]/(X^2)$. Also, we recall that $diam(\Gamma (R))\leq 3$ (cf. theorem 2.3, \cite{AnderLiv})  whose next lemma
is an immediate consequence.
\end{remark}

\begin{lemma} Let $x,y\in Z(R)^{\star}$. If $d_{\Gamma}(x,y)>2$, then $d_{\Gamma}(x,y)=3$.
\end{lemma}

Using the new graph $\widetilde{\Gamma}(R)$, we obtain some cases where $diam(\Gamma (R))=3$:

\begin{theorem}If $\widetilde{\Gamma}(R)$ is not complete, then $diam(\Gamma (R))=3$.
\end{theorem}
\begin{proof}Since $\widetilde{\Gamma}(R)$ is not complete, then $diam(\widetilde{\Gamma}(R))=2$
((cf. \cite{Groupe}, theorem 2.1),
so there exists $x,y\in Z(R)^{\star}$ such that $d_{\widetilde{\Gamma}}(x,y)=2$ hence
$xy\neq 0$ and $x+y\notin Z(R)$ thus $ann(x)\cap ann(y)=(0)$ therefore $d_{\Gamma}(x,y)>2$ and
thus, by the previous lemma, $diam(\Gamma (R))=3$.
\end{proof}

\begin{corollary}If $Z(R)$ is not an ideal of $R$ and $R$ is neither boolean nor (up to isomorphism)
 a subring of a product of two integral domains, then $diam(\Gamma (R))=3$.
\end{corollary}
\begin{proof}
 Since $Z(R)$ is not an ideal of $R$ and $R$ is neither boolean nor a subring of a product of two
 integral domains, then, by theorem 1.7 \cite{Groupe2},  $\widetilde{\Gamma}(R)$ is not complet and,
 by the previous theorem,  $diam(\Gamma (R))=3$.
\end{proof}

\begin{remark} The previous theorem gives a method to construct graphs $\Gamma (R)$ of diameter $3$:
for example,  $diam(\Gamma (\mathbb{Z}_{12}))=3$ because $Z(\mathbb{Z}_{12})$ is not an ideal
($2+3\notin Z(\mathbb{Z}_{12})$) and $ \mathbb{Z}_{12}$ is neither boolean ( $\mathbb{Z}_{12}$ is
not isomorph to $\mathbb{Z}_{2}^n$) nor  a subring of a product of two integral domains ($\mathbb{Z}_{12}$
is not reduced).
\end{remark}

We know that $\widetilde{\Gamma}(R)$ is not complete if and only if $Z(R)$ is not an ideal of $R$ and
$R$ is neither boolean nor (up to isomorphism) a subring of a product of two integral domains
(cf. \cite{Groupe2}, theorem 1.7) so it is enough to treat the cases where $\widetilde{\Gamma}(R)$
is  complete to give a ring characterizations such that $diam(\Gamma (R))=1$, $2$ or $3$, i.e.,
the cases where $Z(R)$ is an ideal of $R$ or $R$ is boolean or $R$ is (up to isomorphism) a subring
of a product of two integral domains.

We have the following preliminary lemma:

\begin{lemma} \hfill
\begin{enumerate}[(1)]
  \item If $Z(R)^2=(0)$, then $Z(R)$ is an ideal.
  \item Let $R$ such that $Z(R)$ is an ideal. If $Z(R)^2\neq (0)$, then there exist a
  distinct pair of non-zero-divisors $x,y$ such that $xy\neq 0$.
  \item Let $R$ such that $Z(R)$ is an ideal. If there exist a  pair of zero-divisors
  $x,y$ such that $ann(x,y)=(0)$, then $x,y$ is distinct pair of non-zero-divisors such that $xy\neq 0$.
\end{enumerate}
\end{lemma}
\begin{proof}(1)  Suppose that $Z(R)^2=(0)$ so $Z(R)\subset Nil(R)$, where $Nil(R)$ is the nilradical of
$R$, then  $Z(R)= Nil(R)$ hence $Z(R)$ is an ideal.\\
(2) Let $x\in Z(R)^{\star}$ such that $x^2\neq 0$. It is clear that if $2x\neq 0$, then $x,-x $ is
a distinct pair of non-zero-divisors $x,y=-x$ such that $xy\neq 0$.
Suppose that $2x=0$ and let $a\in Z(R)^{\star}$ such that $ax=0$. Let $y=a+x$ so $y\in Z(R)$
because $Z(R)$ is an ideal. Also,  $y\neq x$ and $yx=(a+x)x=x^2\neq 0$ then  $x,y$ is a distinct pair
of zero-divisors such that $xy\neq 0$ thus $diam(\Gamma(R))>1$ therefore $diam(\Gamma(R))=2$
because for each  pair  of zero-divisors $x,y$, $ann(x,y)\neq (0)$.\\
(3) Suppose that there exist a  pair of zero-divisors $x,y$ such that
 $ann(x,y)=(0)$ so $x\neq 0$, $y\neq 0$ and $x\neq y$. Also, $x+y\in Z(R)$ because $Z(R)$ is an ideal so
 there exist $a\in R\setminus \{0\}$ such that $a(x+y)=0$ then $ax=-ay$. We claim
 that $xy\neq 0$, indeed, if $xy=0$, so $(ax).x=-ayx=0$ and $(ax)y=0$ hence $ax\in ann(x,y)=(0)$. Also,
 $(ay)x=0$ and $ (ay)y=-axy=0$ then $ay\in ann(x,y)=(0)$ therefore $a\in ann(x,y)=(0)$. \\
\end{proof}

\begin{proposition}\hfill
\begin{enumerate}[(1)]
  \item Let $R$ such that $Z(R)$ is an ideal and $Z(R)^2\neq (0)$.
  If for each distinct pair of zero-divisors $x,y$, $ann(x,y)\neq (0)$, then $diam(\Gamma(R))=2$.
  \item Let $R$ such that $Z(R)$ is an ideal. If there exist a  pair of zero-divisors
  $x,y$ such that $ann(x,y)=(0)$, then $diam(\Gamma(R))=3$.
\end{enumerate}
\end{proposition}
\begin{proof}(1) By lemma 1.7, there exist a distinct pair of
  zero-divisors $a,b$ such that $ab\neq 0$ then $diam(\Gamma(R))>1$. Let $x,y\in Z(R)^{\star}$ such that
$d_\Gamma (x,y)>1$ so  $ann(x,y)\neq (0)$ hence $d_\Gamma (x,y)=2$.\\
(2) Suppose that there exist a  pair of zero-divisors
  $x,y$ such that $ann(x,y)=(0)$, then, by the previous lemma, $x,y$ is distinct pair
  of non-zero-divisors such that $xy\neq 0$ so $diam(\Gamma(R))>1$ and since $ann(x,y)= (0)$,
  $diam(\Gamma(R))>2$ then $diam(\Gamma(R))=3$.
\end{proof}

\begin{remark}Let $R$ such that $Z(R)$ is an ideal. By the previous proof, if there
exist a  pair of zero-divisors $x,y$ such that $ann(x,y)=(0)$ so $xy\neq 0$ then
then $Z(R)^2\neq 0$.
\end{remark}

\begin{proposition} If $R$ is (up to isomorphism)  a subring of a product of two integral domains
and $R\not \simeq \mathbb{Z}_2^2$, then $diam(\Gamma(R))=2$.
\end{proposition}

\begin{proof}Since  $R$ is a subring of a product of two integral domains and $R$ is not an integral
domain, there exists $a=(a_1,0),b=(0,a_2)\in Z(R)^{\star}$. \\
We claim that $|Z(R)^{\star}|\geq 3$, indeed, if $|Z(R)^{\star}|= 2$, $R\simeq \mathbb{Z}_9$ or $\mathbb{Z}_2^2$ or $\mathbb{Z}_3[X]/(X^3)$ then $R\simeq \mathbb{Z}_9$ or $\mathbb{Z}_3[x]/(x^3)$ (because $R\not \simeq \mathbb{Z}_2^2$). However, $\mathbb{Z}_9$ and $\mathbb{Z}_3[x]/(x^3)$ are not reduced but $R$ is reduced then $|Z(R)^{\star}|\geq 3$.\\
Let $x\in Z(R)^{\star}\setminus \{a,b\}$ and suppose that $x=(x_1,0)$ (the other case is similar). Since $x\neq a$ and $ax\neq 0$ so $diam(\Gamma(R))>1$. Also, let $z,t\in Z(R)^{\star}$ such that $d_{\Gamma(R)}(z,t)>1$ so we can suppose that $z=(z_1,0)$ and $t=(t_1,0)$ then $z\--b\--t$ hence $d_{\Gamma}(z,t)=2$ and thus $diam(\Gamma(R))=2$.\\
\end{proof}

\begin{proposition}Let $R$ be a boolean ring. If  $R\not \simeq \mathbb{Z}_2^2$, then $diam(\Gamma(R))=3$.
\end{proposition}

\begin{proof} Since $R$ is boolean and $R\not \simeq \mathbb{Z}_2$, there exists
$e\in R\setminus \{0,1\}$ such that $R\simeq Re\oplus R(1-e)$. Also, since
$R\not \simeq \mathbb{Z}_2^2$, we can suppose that $Re \not \simeq \mathbb{Z}_2$ thus,
since $Re$ is boolean,  there exists $e'\in Re\setminus \{0,1\}$ such that
$Re\simeq Re'\oplus R(1-e')$ therefore we can suppose that $R\simeq R_1 \oplus R_2\oplus R_3$,
with $R_1,R_2,R_3$ boolean rings. Let $x=(1,1,0),y=(1,0,1)\in Z(R)^{\star}$ so $x\neq y$, $xy\neq 0$
and $ann(x)\cap ann(y)=(0)$ then $d_{\Gamma(R)}(x,y)>2$ hence, by  lemma 1.3, $diam(\Gamma(R))=3$.
\end{proof}

\begin{theorem}\hfill
\begin{enumerate}[(1)]
  \item  $diam(\Gamma(R))=1$ if and only if $R \simeq \mathbb{Z}_2^2$ or  $Z(R)^2= (0)$.
  \item $diam(\Gamma(R))=2$ if and only if ($R$ is (up isomorphism) a subring of a product of two
  integral domains and $R\not \simeq \mathbb{Z}_2^2$) or ($Z(R)$ is an ideal, $Z(R)^2\neq (0)$ and
  for each distinct pair of zero-divisors $x,y$, $ann(x,y)\neq (0)$.
  \item $diam(\Gamma(R))=3$  if and only if ($R$ is boolean and $R\not \simeq \mathbb{Z}_2^2$) or
  ($Z(R)$ is not an ideal and $R$ is neither boolean nor a subring of a product of
  two integral domains) or ($Z(R)$ is  an ideal and there
  there exist a  pair of zero-divisors $x,y$ such that $ann(x,y)=(0))$.
\end{enumerate}
\end{theorem}
\begin{proof}Suppose that $Z(R)$ is not an ideal and $R$ is neither boolean nor a subring of a product
of two integral domains. Then, according to theorem 1.7 \cite{Groupe2}, $\widetilde{\Gamma}(R)$
is not complete and thus, by theorem 1.4, $diam(\Gamma(R))=3$.\\
Suppose that $R$ is a boolean ring. It is obvious that if $R \simeq \mathbb{Z}_2^2$, then $(\Gamma(R))$
is complete. If $R\not \simeq \mathbb{Z}_2^2$, then, by proposition 1.11, $diam(\Gamma(R))=3$.\\
Suppose that $R$ is a subring of a product of two integral domains and $R\not \simeq \mathbb{Z}_2^2$,
then, by proposition 1.10, $diam(\Gamma(R))=2$.\\
Suppose that $Z(R)$ is an ideal of $R$: \\
It is obvious that if $Z(R)^2=(0)$, then $diam(\Gamma(R))=1$.\\
Suppose that $Z(R)^2\neq (0)$ then, by proposition 1.8,  if for each distinct pair of zero-divisors
 $x,y$, $ann(x,y)\neq (0)$, $diam(\Gamma(R))=2$. \\
 If $Z(R)$ is an ideal and there exist a  pair of zero-divisors
 $x,y$ such that $ann(x,y)=(0)$, then by proposition 1.8,  $diam(\Gamma(R))=3$. Also, we recall that, by remark
  1.9, $Z(R)^2\neq (0)$.
\end{proof}

We recall that $R$ is a McCoy ring (or satisfy the property A) (cf. \cite{Huck})
if each finitely generated ideal contained in $Z(R)$ has a nonzero annihilator.
\begin{corollary}Let $R$ be a McCoy ring.
\begin{enumerate}[(1)]
  \item $diam(\Gamma(R))=1$ if and only if $R \simeq \mathbb{Z}_2^2$ or  $Z(R)^2= (0)$.
  \item $diam(\Gamma(R))=2$ if and only if ($R$ is (up isomorphism) a subring of a product of two integral
  domains and $R\not \simeq \mathbb{Z}_2^2$) or ($Z(R)$ is an ideal, $Z(R)^2\neq (0))$.
  \item $diam(\Gamma(R))=3$  if and only if ($R$ is boolean and $R\not \simeq \mathbb{Z}_2^2$) or
  ($Z(R)$ is not an ideal and $R$ is neither boolean nor a subring of a product of two integral domains).
  \end{enumerate}
\end{corollary}
\begin{proof} Suppose that $Z(R)$ is an ideal of $R$ such that $Z(R)^2\neq (0)$. Let
a distinct pair of zero-divisors $x,y$ so $(x,y)\subset Z(R)$ because $Z(R)$ is an ideal and since $R$ is a McCoy ring,
$ann(x,y)\neq 0$.
\end{proof}

\begin{lemma} $R$ is a noetherian boolean ring if and only if $R\simeq \mathbb{Z}_2^n$.
\end{lemma}
\begin{proof}
$\Rightarrow)$ Since $R$ is boolean, then $\dim R=0$ so $R$ is artinian hence $R$ has a finite number of maximal  ideals $\mathfrak{m}_1,\dots, \mathfrak{m}_n$. Since $R$  is boolean, $R$ is reduced then $\bigcap\limits_{i=1}^{n}\mathfrak{m}_i=(0)$ so $R\simeq \prod\limits_{i=1}^{n} R/\mathfrak{m}_i$ therefore $R\simeq \mathbb{Z}_2^n$ because $R/\mathfrak{m}_i$ are boolean fields. The other implication is obvious.
\end{proof}

Since a noetherain ring is a McCoy ring (cf. theorem 82, \cite{Kap}), using the previous lemma,  we obtain:

\begin{corollary}Let $R$ a noetherian ring.
\begin{enumerate}[(1)]
  \item $diam(\Gamma(R))=1$ if and only if $R \simeq \mathbb{Z}_2^2$ or  $Z(R)^2= (0)$.
  \item $diam(\Gamma(R))=2$ if and only if ($R$ is (up isomorphism) a subring of a product of
  two integral domains and $R\not \simeq \mathbb{Z}_2^2$) or ($Z(R)$ is an ideal, $Z(R)^2\neq (0)$.
  \item $diam(\Gamma(R))=3$  if and only if ($R \simeq \mathbb{Z}_2^n$, with $n>2$) or ($Z(R)$ is not an ideal and $R$ is neither  $\mathbb{Z}_2^n$, with $n>2$ nor a subring of a product of two integral domains).
  \end{enumerate}
\end{corollary}

Using theorem 2.4 \cite{Groupe}, we obtain when $R$ is a finite ring:
\begin{corollary} Let $R$ be a finite ring.
\begin{enumerate}[(1)]
  \item $diam(\Gamma(R))=1$ if and only if $R \simeq \mathbb{Z}_2^2$ or ($R$ is local and $\mathfrak{m}^2=(0)$).
  \item $diam(\Gamma(R))=2$ if and only if  $R$ is  a product of two fields or  ($R$ is local and
  $\mathfrak{m}^2\neq (0)$.
  \item $diam(\Gamma(R))=3$ if and only if ($R\not \simeq \mathbb{Z}_2^n$  and $R$ is
  neither  a product of two fields nor local) or ($R \simeq \mathbb{Z}_2^n$, with $n>2$).
\end{enumerate}
\end{corollary}

\begin{corollary}Let $n>1$ a composite integer.
\begin{enumerate}[(1)]
  \item $diam(\Gamma(\mathbb{Z}_n))=0$ if and only if $n=4$.
  \item  $\Gamma(\mathbb{Z}_n)=1$  If and only if if $n=p^2$ with $p$ is an odd prime.
  \item $diam(\Gamma(\mathbb{Z}_n))=2$ if and only if $n=p^k$ with $k>2$ and $p$ is prime or
  $n$ is a product of two distinct primes.
  \item $diam(\Gamma(\mathbb{Z}_n))=3$ if and only if $n$ is neither a power of a prime number
  nor a  product of two distinct primes.
\end{enumerate}
\end{corollary}

%\begin{remark} The comparison between lucas's results cited in \cite{Lucas} and our results allows us
%to deduce from them the following algebraic results: we recall that $R$ is a ring such that
%$R\not\simeq \mathbb{Z}_4$ and $R\not\simeq \mathbb{Z}_[X]/(x^2)$.
%\begin{enumerate}[(1)]
 % \item $xy=0$ for each distinct pair  of zero divisors if and only if
  %$R \simeq \mathbb{Z}_2^2$ or  $Z(R)^2= (0)$.
  %\item $R$ is reduced with exactly two minimal primes and at least three nonzero zero divisors if and only if
  %$R$ is (up isomorphism) a subring of a product of two
  %integral domains and $R\not \simeq \mathbb{Z}_2^2$.
  %\item if and only if there are zero divisors $a\neq b$ such that $(0:(a,b))=(0)$
%and either (i) $R$ is a reduced ring with more than two minimal primes, or (ii) $R$ is
%nonreduced if and only if ($R$ is boolean and $R\not \simeq \mathbb{Z}_2^2$) or
 % ($Z(R)$ is not an ideal and $R$ is neither boolean nor a subring of a product of
  %two integral domains) or ($Z(R)$ is  an ideal, $Z(R)^2\neq (0)$ and there exist distinct pair of

  %$(x,y)\in X(R)$
  %such that $ann(x,y)=0$).
%\end{enumerate}
%\end{remark}

\section{examples}
in this section, we give examples of the situations described in the  theorem.
We begin by giving an example where $diam(\Gamma(R))=1$.

\begin{example}Let $R=\mathbb{R}[X]/(X^2)$. It is obvious
  that $Z(R)=(X+(X^2))$ and $Z(R)^2=(0)$ then $diam(\Gamma(R))=1$.
\end{example}

For the case where $diam(\Gamma(R))=2$, we give the following two examples:
\begin{example}Let $R=\mathbb{Z}^2$ so $diam(\Gamma(R))=2$.
\end{example}
%\begin{example}Let $R=k[X,Y]/(XY)$, where where $k$ is a field. It is obvious that  $R$ is a subring  of the two
 %   integral domains $k[X,Y]/(X)$ and $k[X,Y]/(Y)$ so $diam(\Gamma(R))=2$. However, $R$
  %  is not a product of two integral domains.
%\end{example}
\begin{example}Let $R=k[X,Y]/(X^2,XY)$. It is obvious that the $Z(R)=(X+(X^2,XY),Y+(X^2,XY))$
is an ideal of $R$ and since $Y+(X^2,XY)\in Z(R)$ and  $(Y+(X^2,XY))^2\neq 0$ in $R$, $Z(R)^2\neq (0)$.
Also $R$ is noetherian so, by corollary 1.14, $diam(\Gamma(R))=2$.
\end{example}

For the case where $diam(\Gamma(R))=3$, we give also the following three examples:
\begin{example}Let $R=\mathbb{Z}_{2}^{n}$, where $n>2$ is an integer, so $R$ is boolean then $diam(\Gamma(R))=3$.
\end{example}
\begin{example}Let $R=\mathbb{Z}^3$. It is obvious that $Z(R)$ is not an ideal and $R$ is neither
boolean nor a subring of a product of two integral domains hence $diam(\Gamma(R))=3$.
\end{example}
\begin{example}
As in \cite{Lucas}, we will use a variation of the construction $"A+B"$ described in
\cite{Huck} and \cite{Ander} to give an example of a ring $R$ such that $Z(R)$ is  an ideal and
there exist a  pair of zero-divisors $r,s$ such that $ann(r,s)=0$ (then by remark 1.9, $Z(R)^2\neq (0)$):
 Let $M$ the maximal ideal of $A=k[X,Y]_{(X,Y)}$ and
$\mathcal{P}=\{\mathfrak{p}_{\alpha}/\alpha\in \Gamma\}$ the set of height one primes of $A$.  For
every $i=(\alpha,n)\in I=\Gamma \times \mathbb{N}$, let $\mathfrak{p}_i=\mathfrak{p}_{\alpha}$ and
$M_i=M/\mathfrak{p}_i$. It is obvious that $B=\bigoplus\limits_{i\in I} M_i$ is a non-unital ring and
is a unitary $A$-module. As in theorem 2.1 \cite{Ander}, define on $R=A\times B$:
$(a,x)+(b,y)=(a+b,x+y)$ and $(a,x)(b,y)=(ab,ay+bx+xy)$ then $R$ is a commutative ring
    with identity $1_R=(1,0)$ and is noted $R=A+B$.\\
    We claim that $Z(R)=\{(m,b)/m\in M, b\in B\}$ and consequently $Z(R)$ is an ideal:
    Let $(a,x)\in R$ such that $a\notin M$ and $x=(x_i+\mathfrak{p}_i)_{i\in I}\in B$ so
    $\forall m\in M$, $a+m\not\in in M$ and since $A$ is local and $M$ is
    the maximal ideal of $A$, $a+m$ is unit in $A$. For every $i\in I$,
    let $y_i=-a^{-1}(a+x_i)^{-1}x_i$ so $y=(y_i+\mathfrak{p}_i)\in B$ and we have $(a,x)(a^{-1},y)=1_R$
    hence $(a,x)\not \in Z(R)$.
    Conversely, let $(a,x)\in R$ such that $a \in M$ and $x=(x_i+\mathfrak{p}_i)_{i\in I}\in B$.
    It follows from the Krull's principal ideal theorem   that there exist
    $\beta \in \Gamma$ such that $a\in \mathfrak{p}_{\beta}$ so there exist $j\in I$
    such that $a\in \mathfrak{p}_j$ and $x\in \mathfrak{p}_j$ (because $\{i\in I/a\in \mathfrak{p}_i\}$
    is infinite and $\{i\in I/x_i\not \in \mathfrak{p}_i\}$ is finite). Let $v\in M\setminus \mathfrak{p}_j$
    and $y_i=\left\lbrace \begin{array}{c}
v\ \text{si}\ i=j \\
0 \ \text{si}\ i\neq j
\end{array}
\right.$ so $y=(y_i+\mathfrak{p}_j)\in B\setminus\{0\}$ and $(a,x)(0,y)=0$ thus $(a,x)\in Z(R)$.\\
Also, we claim that there exist $(r,s)\in Z(R)^2$ such that $ann(r,s)=(0)$: let $r=(X,0)$ and
$s=(Y,0)$. If  $(a,x)\in  ann(r)\cap ann(s)$, where $a\in A$ and $x=(x_i+\mathfrak{p}_i)_{i\in I}\in B$,
so $r(a,x)=s(a,x)=0$ then
$a=0$ and $\forall i\in I$, $Xx_i\in \mathfrak{p}_i$ and $Yx_i\in \mathfrak{p}_i$ then $\forall i\in I$,
$x_i\in \mathfrak{p}_i$, if not $\exists j\in I$ such that $x_j\not\in \mathfrak{p}_j$ so
$M=(X,Y)= \mathfrak{p}_j$, contradiction, because $ht(M)=2$. Thus $a=0$ and $x=0$.\\
By the previous theorem, we obtain $diam(\Gamma(R))=3$.

\end{example}

\section{ diameter of $\Gamma (R[X])$}

Lucas  gave a following  characterization of the diameter of $\Gamma (R[X])$ (see theorem 3.4,
 \cite{Lucas} ):
\begin{theorem}Let $R$ be a ring.
\begin{enumerate}[(1)]
  \item $diam (R[X])\geq 1$.
  \item $diam (R[X])=1$ if and only if $R$ is a nonreduced ring such that $Z(R)^2= (0)$.
  \item $diam (R[X])=2$ if and only if either (i) $R$ is a reduced ring with exactly two minimal
primes, or (ii) $R$ is a McCoy ring and $Z(R)$ is an ideal with $Z(R)^2\neq (0)$.
  \item $diam (R[X])=3$ if and only if $R$ is not a reduced ring with exactly two minimal
primes and either $R$ is not a McCoy ring or $Z(R)$ is not an ideal.
\end{enumerate}
\end{theorem}

In this section, we will use the results of the study of the graph
$\widetilde{\Gamma}(R[X])$ \cite{Groupe2} to approach the same problem.
We recall that $R[X]$ is a McCoy ring (cf. Theorem 2.7, \cite{Huck}).
We note also  that $R[X]$ is not boolean and if $R$ is not an integral domain, then $|Z(R[X])|>2$.\\

\begin{lemma} $Z(R[X])^2=0$ if and only if $Z(R)^2=0$.
\end{lemma}
\begin{proof}(1) Since $ Z(R)\subset Z(R[X])$, if $Z(R[X])^2=0$ then $Z(R)^2=0$.
Conversely, since $Z(R[X])\subset (Z(R))[X]$ (cf. Exercise 2, iii), page 13, \cite{A}), if $Z(R)^2=0$, then $Z(R[X])^2=0$.
\end{proof}

Using corollary 1.13 and the previous lemma, we obtain:

\begin{theorem}Let $R$ a ring such that $R$ is not an integral domain.
\begin{enumerate}[(1)]
  \item  $diam(\Gamma(R[X]))=1$ if and only if  $Z(R)^2= (0)$.
  \item $diam(\Gamma(R[X]))=2$ if and only if ($R$ is (up isomorphism) a subring of a product
  of two integral domains   or ($R$ is a McCoy ring and $Z(R)$ is an ideal such that $Z(R)^2\neq (0)$).
  \item $diam(\Gamma(R[X]))=3$  if and only if   ($R$ is not a McCoy ring or
  $Z(R)$ is not an ideal) and $R$ is not a subring of a product of two integral domains.
\end{enumerate}
\end{theorem}
\begin{proof}
\begin{enumerate}[(1)]
  \item The result is a consequence of the lemma 3.2 and the fact that $R[X]\not \simeq  \mathbb{Z}_2^2$.
  \item By corollary 1.13, $diam(\Gamma(R[X]))=2$ if and only if $R[X]$ is (up isomorphism) a subring of
  a product of two integral domains and $R[X]\not \simeq \mathbb{Z}_2^2$) or ($Z(R[X])$ is an ideal,
  $Z(R[X])^2\neq (0)$. It is obvious that $R[X]$ is (up isomorphism) a subring of a product of
  two integral domains if and only if $R$ is (up isomorphism) a subring of a product of two integral domains.
  On the other hand, by lemma 1.10 \cite{Groupe2}, $Z(R[X])$ is an ideal such that $Z(R[X])^2\neq (0)$
  if and only if $R$ is a McCoy ring and $Z(R)$ is an ideal such that $Z(R)^2\neq (0)$.
  \item Also, by corollary 1.13, $diam(\Gamma(R[X]))=3$  if and only if ($R[X]$ is boolean and
  $R[X]\not \simeq \mathbb{Z}_2^2$) or ($Z(R[X])$ is not an ideal and $R[X]$ is neither boolean
  nor a subring of a product of two integral domains).
  It is obvious that $R[X]$ is not boolean and $R[X]$ is not a subring of a product
  of two integral domains if and only if $R$ is not a subring of a product of two integral domains.
  Also, by lemma 1.10 \cite{Groupe2}, $Z(R[X])$ is not an ideal if and only if $R$ is not a
  McCoy ring or $Z(R)$ is not an ideal.
\end{enumerate}
\end{proof}

\begin{remark}We recall that, by lemma 1.9 \cite{Groupe2}, $Z(R[X])$ is an ideal of $R[X]$ if and only if
$Z(R)$ is an ideal of $R$ and $R$ is a McCoy ring if and only if for any ideal $I$  of $R$ generated by a
finite number of zero-divisors, $ann(I)\neq (0)$.\\
If $R$ is noetherian so $R$ is a McCoy ring, then $Z(R[X])$ is an ideal of $R[X]$ if and only if
$Z(R)$ is an ideal of $R$.
\end{remark}

 \begin{corollary}Let $R$ a noetherian ring such that $R$ is not an integral domain.
 \begin{enumerate}[(1)]
  \item  $diam(\Gamma(R[X]))=1$ if and only if  $Z(R)^2= (0)$.
  \item $diam(\Gamma(R[X]))=2$ if and only if ($R$ is (up isomorphism) a subring of a product of two integral domains   or ($Z(R)$ is an ideal and $Z(R)^2\neq (0)$).
  \item $diam(\Gamma(R[X]))=3$  if and only if   $Z(R)$ is not an ideal and $R$ is neither boolean nor a subring of a product of two integral domains.
\end{enumerate}
 \end{corollary}

 \begin{corollary}Let $R$ a finite ring such that  $R$ is not a field.
 \begin{enumerate}[(1)]
  \item  $diam(\Gamma(R[X]))=1$ if and only if  $R$ is local and $\mathfrak{m}^2=(0)$.
  \item $diam(\Gamma(R[X]))=2$ if and only if $R$ is  a product of two fields or
  ($R$ is local and $\mathfrak{m}^2\neq (0)$.
  \item $diam(\Gamma(R[X]))=3$  if and only if   $R$ is not local and $R$ is not
   a product of two fields.
\end{enumerate}
 \end{corollary}

\begin{corollary}Let $n>1$ a composite integer.
\begin{enumerate}[(1)]
  \item  $\Gamma(\mathbb{Z}_n[X])=1$  If and only if if $n=p^2$ with $p$ is prime.
  \item $diam(\Gamma(\mathbb{Z}_n[X]))=2$ if and only if $n$ is a
  product of two distinct primes or $n=p^k$ with $k>2$ and $p$ is prime.
  \item $diam(\Gamma(\mathbb{Z}_n[X]))=3$ if and only if $n$ is neither a power of a prime number
  nor a  product of two distinct primes.
\end{enumerate}
\end{corollary}

%%%%%%%%%%%%%%%%%%%%%%%%%%%%%%%%%%%%%%%%%%%%%%%%%%%%%%%%%%%%%%%%%%%%%%%%%%%%%%%%%%%%%%%%%%%%%%%%%%%%%%
%%%%%%%%%%%%%%%%%%%%%%%%%%%%%%%%%%%%%%%%%%%%%%%%%%%%%%%%%%%%%%%%%%%%%%%%%%%%%%%%%%%%%%%%%%%%%%%%%%%%%%%

\end{sloppypar}

\end{document}